\documentclass[a4paper,10pt]{scrartcl}

\usepackage{amsmath, amstext, amsthm, amsopn, amssymb, amscd, amsxtra, amsfonts}
\usepackage{epsfig, floatflt, setspace, tocloft, paralist, manfnt, graphicx, longtable}
\usepackage[mathcal]{euscript}
\usepackage{diagrams}

\newfont{\theoremfont}{cmssbx12 scaled 875}
\newfont{\thminlinefont}{cmti12 scaled 850}

\newtheoremstyle{PLAIN}{\topsep}{\topsep}{}{}{\theoremfont}{.}{5pt}{\thmname{#1}\thmnumber{ #2}\thmnote{ #3}}
\theoremstyle{PLAIN}
\newtheorem*{thmA}{Theorem A}
\newtheorem*{thmB}{Theorem B}
\newtheorem{thm}{Theorem}[section]
\newtheorem{cor}[thm]{Corollary}

\newtheorem{lem}[thm]{Lemma}

\newtheorem{ex}[thm]{Example}
\newtheorem{rem}[thm]{Remark}

\newcommand{\ind}[1]{\operatorname{ind}_{#1}\nolimits}

\newcommand{\proj}[1]{\operatorname{proj}_{#1}\nolimits}
\newcommand{\Bigsum}[2]{\ensuremath{\mathop{\textstyle\sum}_{#1}^{#2}}}
\newcommand{\Bigcup}[1]{\ensuremath{\mathop{\cup}_{#1}}}
\newcommand{\Bigcap}[1]{\ensuremath{\mathop{\cap}_{#1}}}

\renewcommand{\Re}{\operatorname{Re}\nolimits}
\newcommand{\id}{I}
\newcommand{\gr}{\operatorname{gr}\nolimits}
\newcommand{\abs}{\operatorname{abs}\nolimits}
\newcommand{\seq}{\operatorname{seq}\nolimits}

\newcommand{\hatt}{\scriptscriptstyle\wedge}

\renewcommand{\epsilon}{\varepsilon}

\allowdisplaybreaks
\newcommand{\spacea}{\hspace{1.2pt}}

\begin{document}
\setlength{\parskip}{4pt}
\bibliographystyle{amsplain}
\title{Universal extrapolation spaces\\for C$_{\text{0}}$-semigroups}
\author{Sven-Ake Wegner\,$^{1}$}
\date{September 3, 2013}

\maketitle
\renewcommand{\thefootnote}{}
\hspace{-1000pt}\footnote{2010 \emph{Mathematics Subject Classification}: Primary 47D06; Secondary 47D03, 46A13.}
\hspace{-1000pt}\footnote{\emph{Key words and phrases}: semigroup, extrapolation space, interpolation space, So\-bo\-lev tower.}
\hspace{-1000pt}\footnote{$^{1}$\vspace{1pt}Bergische Universit\"at Wuppertal, FB C -- Mathematik, Gau\ss{}stra\ss{}e 20, D-42119 Wuppertal, Germany, Phone:\spacea{}\spacea{}+49\spacea{}(0)\spacea{}202\spacea{}/\spacea{}439\spacea{}-\spacea{}2531, Fax:\spacea{}\spacea{}+49\spacea{}(0)\spacea{}202\spacea{}/\spacea{}439\spacea{}-\spacea{}3771, eMail: wegner@math.uni-wuppertal.de.}

\vspace{-30pt}

\begin{abstract}
The classical theory of Sobolev towers allows for the construction of an infinite ascending chain of extrapolation spaces and an infinite descending chain of interpolation spaces associated with a given $C_0$-semigroup on a Banach space. In this note we first generalize the latter to the case of a strongly continuous and exponentially equicontinuous semigroup on a complete locally convex space. As a new concept -- even for $C_0$-semigroups on Banach spaces -- we then define a universal extrapolation space as the completion of the inductive limit of the ascending chain. Under mild assumptions we show that the semigroup extends to this space and that it is generated by an automorphism of the latter. Dually, we define a universal interpolation space as the projective limit of the descending chain. We show that the restriction of the initial semigroup to this space is again a semigroup and always has an automorphism as generator.
\end{abstract}

\setlength{\parskip}{12pt}

\section{Introduction}\label{Intro}
\vspace{-10pt}
Extrapolation spaces constitute an important technique arising in the theory of evolution equations. Headed by its prominent prototype, the classical Sobolev space $H^{-1}$, extrapolation spaces appear naturally in many situations such as the theory of boundary perturbations, cf.~the survey \cite{Nagel} of Nagel and its list of references, or in the construction of rigged Hilbert spaces, see e.g.~Triebel \cite{Triebel}. Abstract approaches can be found in Engel, Nagel \cite[II.5]{EngelNagelOne} or van Neerven \cite{vanNeervenBook}; they are based on the key notion of so-called Sobolev towers, which provide a method to construct an infinite chain of extrapolation spaces associated with a semigroup acting on a given Banach space.
\smallskip
\\The aim of this paper is twofold. In Section \ref{S1} we generalize the construction of Sobolev towers from the classical setting of $C_0$-semigroups on Banach spaces to those of a strongly continuous and exponentially equicontinuous resp.~quasi-equicontinuous semigroup in the sense of Albanese, Bonet, Ricker \cite{ABR,ABR13} resp.~Choe \cite{Choe} defined on a complete locally convex space. Our purpose is to facilitate the methods of extrapolation spaces in this larger framework for applications in perturbation theory. This is at the moment investigated by Wintermayr \cite{JW}. Let us mention that various notions and results from interpolation theory have been successfully generalized to the setting of locally convex spaces since the mid eighties, see e.g.~Lamb \cite{Lamb}, Lamb, McBride \cite{LambMcBride}, Lamb, Schiavone \cite{SchiavoneLamb}, Mart{\'{\i}}nez, Sanz, Calvo \cite{MSC} or Mart{\'{\i}}nez, Sanz \cite{MS}. Section \ref{Extra} is devoted to a rigorous treatment of what we like to call 
universal extrapolation spaces. Given a Sobolev tower, i.e.~an infinite ascending chain of spaces with continuous inclusion maps and semigroups acting consistently on each of these spaces, we consider the completion of the inductive limit. We show that the semigroups on the steps induce a semigroup with continuous generator on the latter space which extends to an automorphism of the whole space. Both, semigroup and generator, are compatible with the corresponding maps on each step of the tower. Under an additional barrelledness assumption we show that the domain of the generator even coincides with the whole space. We accomplish our investigation in Section \ref{Inter} by a short discussion of the dual construction, namely the projective limit of the descending chain of interpolation spaces. Throughout the article we give examples from the world of sequence spaces. For the convenience of the reader we state our main results on universal extrapolation and interpolation spaces in the setting of a $C_0$-
semigroup acting on a Banach space in Section \ref{S0} right below.
\smallskip
\\Let us mention that Haase \cite[p.~221ff]{Haase} recently used an approach similar to ours, but in a different setting. He considered extrapolation spaces w.r.t.~an injective and continuous operator on a Banach space and defined the corresponding universal extrapolation space as their algebraic inductive limit. On this linear space Haase then defined a notion of net convergence customized for his purposes in \cite{Haase}.
\smallskip
\\For our notation on semigroups and locally convex spaces we refer to the textbooks of Engel, Nagel \cite{EngelNagelOne}, Jarchow \cite{Jarchow}, Meise, Vogt \cite{MeiseVogtEnglisch} and Bonet, P\'{e}rez Carreras \cite{BPC}. Throughout the paper all locally convex spaces are tacitly assumed to be Hausdorff. The definition of an inductive limit thus always includes the assumption that this space is again Hausdorff. In concrete cases this can often be checked directly; for instance the latter holds in our case if there exists some locally convex space that contains all extrapolation spaces with continuous inclusions.

\vspace{-5pt}
\section{The Main Results}\label{S0}
\vspace{-10pt}
Let $X_0$ be a Banach space and $(T_0(t))_{t\geqslant0}$ be a strongly continuous semigroup with generator $A_0$ and assume that $0\in\rho(A_0)$ holds. Let $n\geqslant1$. For $x\in D(A^n)$ we put $\|x\|_n=\|A^nx\|_0$ and $X_n=(D(A^n),\|\cdot\|_n)$. For $x\in X_{-n+1}$ we define recursively $\|x\|_{-n}=\|A_{-n+1}^{-1}x\|_{-n+1}$ and put $X_{-n}=(X_{-n+1},\|\cdot\|_{-n})^{\hatt}$. By $(T_n(t))_{t\geqslant0}$ we denote the restriction resp.~extension of the initial semigroup to $X_n$. By $A_n$ we denote the corresponding generators. We refer to \cite[II.5]{EngelNagelOne} for a detailed study of these so-called \emph{inter-} and \emph{extrapolation spaces of order $n$} in the Banach space case.
\smallskip
\\Remember, that the inductive limit of the spectrum $(X_{-n})_{n\in\mathbb{N}}$ is the union of all $X_{-n}$ endowed with the finest linear topology which makes the inclusions $X_{-n}\rightarrow\Bigcup {k\in\mathbb{N}}X_{-k}$ continuous. The projective limit of the spectrum $(X_n)_{n\in\mathbb{N}}$ is the intersection of all $X_n$ endowed with the coarsest topology which makes all inclusions $\Bigcap{k\in\mathbb{N}}X_k\rightarrow X_n$ continuous.

\begin{thmA}\label{thmA} Assume that the inductive limit $\ind{n\in\mathbb{N}}X_{-n}$ of all extrapolation spaces is Hausdorff and let $X_{-\infty}$ be its completion. Then $X_{-\infty}$ is a complete barrelled locally convex space. The semigroups $T_n$ induce a strongly continuous semigroup $T_{-\infty}$ on $X_{-\infty}$ whose generator $A_{-\infty}\colon X_{-\infty}\rightarrow X_{-\infty}$ is an isomorphism. For $t\geqslant0$ the maps $T_{-\infty}(t)$ resp.~$A_{-\infty}$ are extensions of $T_{-n}(t)$ resp.~$A_{-n}$ for every $n\geqslant1$. We call $X_{-\infty}$ the \emph{universal extrapolation space} associated with the semigroup $(T_0(t))_{t\geqslant0}$.
\end{thmA}

\begin{thmB}\label{thmB} Let $X_{\infty}=\proj{n\in\mathbb{N}}X_n$ be the projective limit of all interpolation spaces. Then $X_{\infty}$ is a Fr\'{e}chet space. The semigroups $T_n$ induce a strongly continuous semigroup $T_{\infty}$ on $X_{\infty}$ whose generator $A_{\infty}\colon X_{\infty}\rightarrow X_{\infty}$ is an isomorphism. For $t\geqslant0$ the maps $T_{\infty}(t)$ resp.~$A_{\infty}$ are restrictions of $T_n(t)$ resp.~$A_n$ for every $n\geqslant1$. We call $X_{\infty}$ the \emph{universal interpolation space} associated with the semigroup $(T_0(t))_{t\geqslant0}$.
\end{thmB}

\vspace{-10pt}

In the sequel we first generalize the construction of inter- and extrapolation spaces of finite order to the setting of locally convex spaces. Then, see Theorem's \ref{t2} and \ref{t3}, we prove results which contain the above statements as special cases.

\vspace{-5pt}
\section{Sobolev Towers on Locally Convex Spaces}\label{S1}
\vspace{-10pt}
Let $(X_0,\tau_0)$ be a complete locally convex space. We assume that the topology $\tau_0$ is given by the system $(p_{\gamma,0})_{\gamma\in\Gamma}$ of seminorms. Let $T_0=(T_0(t))_{t\geqslant0}$ be a strongly continuous and exponentially equicontinuous semigroup on $X_0$. That is, $T_0\colon [0,\infty[\,\rightarrow L(X_0)$ is a map which satisfies the evolution property $T_0(t+s)=T_0(t)T_0(s)$ for $s$, $t\geqslant0$, $T_0(0)=\id_{X_0}$, $T_0(\cdot)x\colon [0,\infty[\,\rightarrow X_0$ is continuous for every $x\in X_0$ and there exists $\omega\in\mathbb{R}$ such that  $\{\exp(\omega t)T_0(t)\:;\:t\geqslant0\}\subseteq L(X_0)$ is equicontinuous. Explicitly, the latter  means that
$$
\exists\:\omega\in\mathbb{R}\;\forall\:\alpha\in\Gamma\;\exists\:\beta\in\Gamma,\,M\geqslant1\;\forall\:t\geqslant0,\,x\in X_0\colon p_{\alpha,0}(T_0(t)x)\leqslant M\,e^{\omega t} p_{\beta,0}(x)
$$
holds. We denote by $(A_0,D(A_0))$ the generator of $T_0$, i.e.
$$
A_0x=\lim_{t\searrow0}{\textstyle\frac{T_0(t)x-x}{t}} \;\text{ for }\; x\in D(A_0)=\{x\in X_0\:;\:\lim_{t\searrow0}{\textstyle\frac{T_0(t)x-x}{t}}\text{ exists}\}.
$$
We refer to K\={o}mura \cite[Section 1]{Komura} for the basic properties of $A_0$, which even hold under weaker assumptions than we required above. In the sequel we assume that $A_0\colon D(A_0)\rightarrow X_0$ is bijective and $A_0^{-1}\colon X_0\rightarrow X_0$ is continuous. As in the Banach space case (cf.~\cite[Remarks previous to Definition II.5.1]{EngelNagelOne}) the latter can always be realized by a rescaling procedure.

\begin{lem}\label{rescaling} Let $(X,\tau)$ be a complete locally convex space and $T=(T(t))_{t\geqslant0}$ be a strongly continuous and exponentially equicontinuous semigroup on $X$ with generator $(A,D(A))$. For $\lambda\in\mathbb{C}$ consider $S=(S(t))_{t\geqslant0}$ with $S(t)=\exp(\lambda t)T(t)$ for $t\geqslant0$. Then $S$ defines a strongly continuous and exponentially equicontinuous semigroup with generator $(B,D(B))$ where $B=A+\lambda$ and $D(B)=D(A)$.
\end{lem}
\vspace{-20pt}
\begin{proof} Clearly, $S$ satisfies the semigroup property, is strongly continuous and also exponentially equicontinuous. Let $(B,D(B))$ be the generator of $S$. For $x\in X$ and $t>0$ we have
\begin{equation*}
{\textstyle\frac{S(t)x-x}{t}}={\textstyle\frac{\exp(\lambda t)T(t)x-x}{t}}=\exp(\lambda t){\textstyle\frac{T(t)x-x}{t}} + {\textstyle\frac{\exp(\lambda t)-1}{t}x},
\end{equation*}
where for $t\searrow0$ the second summand on the right hand side converges to $\lambda x$ for any $x\in X$. If $x$ belongs to $D(A)$ the first summand on the right hand side converges to $Ax$ as the scalar multiplication $\mathbb{C}\times X\rightarrow X$ is continuous. Whence $\lim_{t\searrow0}\frac{S(t)x-x}{t}$ exists, $x\in D(B)$ holds and we have $Bx=Ax+\lambda x$. For $x\in D(B)$ the left hand side above converges to $Bx$ and hence $Bx-\lambda x=\lim_{t\searrow0}\exp(\lambda t)\frac{T(t)x-x}{t}$ exists in $X$. Using a second time that scalar multiplication (with $\exp(-\lambda t)$) is continuous, we obtain that $\lim_{t\searrow0}\frac{T(t)x-x}{t}$ exists. Thus, $x\in D(A)$ and $Ax=Bx-\lambda x$ follow.
\end{proof}

\vspace{-10pt}
For $A_0$ as above there exists $\lambda\in\mathbb{C}$ such that $(A_0+\lambda)^{-1}$ exists and belongs to $L(X_0)$, see \cite[Corollary 4.5 and condition (G) on p.~295]{Choe}. By Lemma \ref{rescaling}, $(B_0,D(B_0))=(A_0+\lambda,D(A_0))$ is the generator of $S_0=(\exp(\lambda t)T_0(t))_{t\geqslant0}$.
\medskip
\\For any $n\in\mathbb{Z}$ we now construct inductively a complete locally convex space $(X_n,\tau_n)$ together with a strongly continuous and exponentially equicontinuous semigroup $T_n=(T_n(t))_{t\geqslant0}$ with generator $(A_n,D(A_n))$. Our results and their proofs are inspired by the theory of extrapolation spaces as presented e.g.~in \cite[II.5]{EngelNagelOne} and rely on the approach of Nagel \cite{NagelTuebingen}, cf.~\cite[remarks on p.~155]{EngelNagelOne}. In the case of Banach spaces, Da Prato, Grisvard \cite{DaPratoGrisvard, DaPratoGrisvard2} gave an alternative definition which does not involve the formation of a completion. However, if we start with a $C_0$-semigroup, both constructions yield isomorphic extrapolation spaces. We refer to \cite[Chapter 3]{vanNeervenBook} for details and remark that it is not clear to us if an adaption of the definition $X^{-1}=(X_0\times X_0)/\gr(A_0)$ of Da Prato, Grisvard (see \cite[p.~41]{vanNeervenBook}) is possible in a straight forward way in our locally 
convex setting -- in particular as the right hand side of the last equation is a priori not complete.
\smallskip
\\For $n\geqslant1$ we put $X_n=D(A_{n-1})$ and define $\tau_n$ to be the topology given by the system $(p_{\gamma,n})_{\gamma\in\Gamma}$ of seminorms with $p_{\gamma,n}(x)=p_{\gamma,n-1}(A_{n-1}x)$ for $x\in X_n$. For $t\geqslant0$ the restriction $T_n(t)=T_{n-1}(t)|_{X_n}\colon X_n\rightarrow X_n$ is well-defined; we put $T_n=(T_n(t))_{t\geqslant0}$. $(A_n,D(A_n))$ denotes the generator of $T_n$.
\smallskip
\\For $n\leqslant-1$ we first define a topology $\check{\tau}_n$ on $X_{n+1}$ via the seminorms $(\check{p}_{\gamma,n})_{\gamma\in\Gamma}$ with $\check{p}_{\gamma,n}(x)=p_{\gamma,n+1}(A_{n+1}^{-1}x)$ for $x\in X_{n+1}$. Now we define $(X_{n},\tau_n)$ to be the completion $(X_{n+1},\check{\tau}_n)^{\hatt{}}$ of $(X_{n+1},\check{\tau}_n)$. For $t\geqslant0$ there exists a unique continuous extension $T_n(t)\colon(X_{n},\tau_n)\rightarrow(X_{n},\tau_n)$ of $T_{n+1}(t)\colon X_{n+1}\rightarrow X_{n+1}$; we put $T_n=(T_n(t))_{t\geqslant0}$. $(A_n,D(A_n))$ is the generator of $T_n$.
\smallskip
\\In analogy to the case of Banach spaces we call $(X_n,\tau_n)$ the \textit{$n$-th Sobolev space} associated with the semigroup $T_0$. In order to prove our main theorem we need the following lemma in which we collect results of Albanese, Bonet, Ricker \cite[Lemma 4.1]{ABR13} and Albanese, K\"uhnemund \cite[Proposition 7]{AK}.

\begin{lem}\label{l1}
\begin{compactitem}
\item[(i)] Let $(X,\tau)$ be a locally convex space and $A\colon D(A)\rightarrow X$ be a closed linear operator with domain $D(A)\subseteq X$. If $(p_{\gamma})_{\gamma\in\Gamma}$ is a fundamental system of seminorms for $\tau$, then the system $(p_{\gamma,A}(\cdot)=p_{\gamma}(\cdot)+p_{\gamma}(A(\cdot)))_{\gamma\in\Gamma}$ defines the \textit{graph topology} $\tau_A$ on $D(A)$. The locally convex space $(D(A),\tau_A)$ is (sequentially) complete if $X$ is (sequentially) complete.\vspace{3pt}
\item[(ii)] Let $(X,\tau)$ be a sequentially complete locally convex space, $(A,D(A))$ be the generator of a strongly continuous and exponentially equicontinuous semigroup $T=(T(t))_{t\geqslant0}$ on $X$. Assume that $D\subseteq(D(A),\tau|_{D(A)})$ is dense and $T$-invariant. Then $D\subseteq (D(A),\tau_{A})$ is dense. \hfill\qed
\end{compactitem}
\end{lem}

\vspace{-10pt}
In the classical case, see e.g.~Engel, Nagel \cite[II.1.6]{EngelNagelOne}, aswell as in the notation of Albanese, K\"uhnemund \cite[Definition 6]{AK} the space $D$ is said to be a \textit{core} for $A$.
\smallskip
\\Let us now prove the main theorem of this section.
\smallskip

\begin{thm}\label{t1} The spaces $(X_n,\tau_n)$ are well-defined and complete locally convex spaces with the following properties.
\begin{compactitem}\vspace{-10pt}
\item[(i)] $X_{n}\subseteq X_{n-1}$ holds and $X_{n}\subseteq(X_{n-1},\tau_{n-1})$ is dense. The inclusion map $(X_{n},\tau_{n})\hookrightarrow (X_{n-1},\tau_{n-1})$ is continuous.\vspace{3pt}
\item[(ii)] For $n\leqslant0$ the topology $\tau_{n}$ is given by the seminorms $(p_{\gamma,n-1}(A_{n-1}(\cdot)))_{\gamma\in\Gamma}$.\vspace{3pt}
\item[(iii)] For $n\geqslant0$ the space $(X_n,\tau_n)$ is the completion of $(X_{n+1},\check{\tau}_{n})$, where $\check{\tau}_{n}$ is given by the seminorms $(\check{p}_{\gamma,n}(\cdot)=p_{\gamma,n+1}(A_{n+1}^{-1}(\cdot)))_{\gamma\in\Gamma}$.
\end{compactitem}\vspace{-5pt}
The families $T_n=(T_n(t))_{t\geqslant0}$ define strongly continuous and exponentially equicontinuous semigroups on the $(X_n,\tau_n)$ with the following properties.
\begin{compactitem}\vspace{-5pt}
\item[(iv)] $T_{n}(t)\colon X_{n}\rightarrow X_{n}$ is the restriction of $T_{n-1}(t)\colon X_{n-1}\rightarrow X_{n-1}$ and the unique continuous extension of $T_{n+1}(t)\colon X_{n+1}\rightarrow X_{n+1}$ for every $t\geqslant0$.\vspace{3pt}
\item[(v)] $D(A_n)=X_{n+1}$ holds and the generator $A_n\colon(X_{n+1},\tau_{n+1})\rightarrow(X_n,\tau_n)$ is an isomorphism. It is the restriction of $A_{n-1}\colon X_{n}\rightarrow X_{n-1}$ and the unique continuous extension of $A_{n+1}\colon X_{n+2}\rightarrow X_{n+1}$.
\end{compactitem}
\end{thm}
\vspace{-20pt}
\begin{proof} $(X_0,\tau_0)$ is complete, $T_0$ is a strongly continuous and exponentially equicontinuous semigroup with generator $A_0$ and  $A_0^{-1}$ is continuous as a map from $(X_0,\tau_0)$ into itself by definition. We show by induction that
\begin{compactitem}\vspace{-3pt}
\item[\text{$[n]$}] $(X_n,\tau_n)$ is complete, $X_n\subseteq (X_{n-1},\tau_{n-1})$ is dense and the inclusion map $(X_n,\tau_n)\hookrightarrow(X_{n-1},\tau_{n-1})$ is continuous. $T_n$ is well-defined and a strongly continuous and exponentially equicontinuous semigroup. $A_n^{-1}\colon(X_n,\tau_n)\rightarrow(X_n,\tau_n)$ exists and is continuous. $\check{p}_{\gamma,n-1}=p_{\gamma,n-1}$ holds for all $\gamma\in\Gamma$.  
\end{compactitem}\vspace{-3pt}
holds for all $n\geqslant1$. Let us show in one step $[1]$ and $[n-1]\Rightarrow[n]$ for $n\geqslant2$.
\smallskip
\\(1) The identity $\id_{D(A_{n-1})}\colon(D(A_{n-1}),\tau_{A_{n-1}})\rightarrow(D(A_{n-1}),\tau_{n})$ is continuous as $p_{\gamma,n}(x)\leqslant p_{\gamma,A_{n-1}}(x)$ holds for any $\gamma\in\Gamma$ and $x\in D(A_{n-1})$, see Lemma \ref{l1}.(i) for the notation. On the other hand for $\alpha\in\Gamma$ there exist $\beta$, $\gamma\in\Gamma$ and constants $C$, $C'\geqslant0$ such that
\begin{eqnarray*}
p_{\alpha,A_{n-1}}(x) \hspace{-6pt}&\leqslant&\hspace{-6pt} p_{\alpha,n-1}(A_{n-1}^{-1}A_{n-1}x)+p_{\alpha,n-1}(A_{n-1}x)\\
                                   &\leqslant&\hspace{-8pt} C p_{\beta,n-1}(A_{n-1}x)+p_{\alpha,n-1}(A_{n-1}x)\leqslant C' p_{\gamma,n}(x)
\end{eqnarray*}
holds for any $x\in D(A_{n-1})$, since $A_{n-1}^{-1}\colon (X_{n-1},\tau_{n-1})\rightarrow(X_{n-1},\tau_{n-1})$ is continuous. Thus, $\id_{D(A_{n-1})}$ is an isomorphism and $(D(A_{n-1}),\tau_{n})$  is complete since this is true for $(D(A_{n-1}),\tau_{A_{n-1}})$, see Lemma \ref{l1}.(i). By \cite[Proposition 1.3]{Komura} we have that $X_n=D(A_{n-1})\subseteq(X_{n-1},\tau_{n-1})$ is dense. Using a second time that $A_{n-1}^{-1}$ is continuous it follows that $(X_n,\tau_n)\hookrightarrow(X_{n-1},\tau_{n-1})$ is continuous.
\smallskip
\\(2) By \cite[Proposition 1.2.(1)]{Komura}, $T_n(t)\colon X_n\rightarrow X_n$ is well-defined for any $t\geqslant0$ and then automatically satisfies the semigroup property. Since $T_{n-1}$ is exponentially equicontinuous, there exists $\omega\in\mathbb{R}$ such that for any $\alpha\in\Gamma$ there exist $\beta\in\Gamma$ and $M\geqslant0$ such that
$$
p_{\alpha,n}(T_n(t)x)=p_{\alpha,n-1}(T_{n-1}(t)A_{n-1}x)\leqslant M\exp(\omega t)p_{\beta,n-1}(A_{n-1}x)=M\exp(\omega t)p_{\beta,n}(x)
$$
holds for all $x\in X_n$ and $t\geqslant0$. Thus, $T_n$ is exponentially equicontinuous. For $s\geqslant0$, $\alpha\in\Gamma$ and $x\in X_n$ we have
$$
p_{\alpha,n}(T_n(t)x-T_n(s)x)=p_{\alpha,n-1}(T_{n-1}(t)A_{n-1}x-T_{n-1}(s)A_{n-1}x)\stackrel{\scriptstyle t\rightarrow s}{\longrightarrow}0
$$
since $T_{n-1}$ is strongly continuous on $(X_{n-1},\tau_{n-1})$.
\smallskip
\\(3) Let $x\in X_n$ be given. By the definition of $\tau_{n}$ and since $(X_n,\tau_n)$ is complete by (1), $\lim_{h\searrow0}\frac{1}{h}(T_n(h)x-x)$ exists in $(X_n,\tau_n)$ if and only if $\lim_{h\searrow0}\frac{1}{h}(T_{n-1}(h)A_{n-1}x-A_{n-1}x)$ exists in $(X_{n-1},\tau_{n-1})$, thus $D(A_{n})=\{x\in X_n\:;\:A_{n-1}x\in X_n\}$ holds. For $x\in D(A_n)$ we have (note that $A_{n-1}x\in X_n=D(A_{n-1})$ follows from the above)
$$
p_{\alpha,n}(A_{n-1}x-{\textstyle\frac{1}{t}}(T_n(t)x-x)) = p_{\alpha,n-1}(A_{n-1}^2x-{\textstyle\frac{1}{t}}(T_{n-1}(t)A_{n-1}x-A_{n-1}x))\stackrel{\scriptstyle t\searrow0}{\longrightarrow}0
$$
since $\lim_{t\searrow0}\frac{1}{t}(T_{n-1}(t)A_{n-1}x-A_{n-1}x)=A_{n-1}^2x$ in $(X_{n-1},\tau_{n-1})$. Thus we have $A_{n}=A_{n-1}|_{X_{n+1}}$ and in particular $A_n$ is injective. Since $X_{n+1}=\{x\in X_n\:;\:A_{n-1}x\in X_n\}$ holds, $A_n\colon X_{n+1}\rightarrow X_n$ is surjective. Thus, $A_{n}^{-1}$ exists and it is continuous as a map from $(X_n,\tau_n)$ to $(X_{n+1},\tau_{n+1})$ by the definition of $\tau_{n+1}$ and thus as a map from $(X_n,\tau_n)$ into itself by (1).
\smallskip
\\(4) By (3), $\check{\tau}_{n-1}$ is well-defined and it follows that $A_{n}^{-1}=A_{n-1}^{-1}|_{X_n}$ holds. Thus for $\gamma\in\Gamma$ and $x\in X_n$ we have $\check{p}_{\gamma,n-1}(x)=p_{\gamma,n}(A_n^{-1}x)=p_{\gamma,n-1}(A_{n-1}A_{n-1}^{-1}x)=p_{\gamma,n-1}(x)$.
\smallskip
\\(5) $A_{n}\colon (X_{n+1},\tau_{n+1})\rightarrow (X_{n},\tau_{n})$ is continuous by the definition of $\tau_{n+1}$. By (3) it is therefore an isomorphism with inverse $A_n^{-1}\colon(X_{n},\tau_{n})\rightarrow(X_{n+1},\tau_{n+1})$.
\bigskip
\\We have shown so far all statements of the theorem which involve only non-negative indices; it remains to prove those which involve at least one strictly negative index. Again we show by induction that $[-1]$ and $[n+1]\Rightarrow[n]$ holds for $n\leqslant-2$ where
\begin{compactitem}\vspace{-3pt}
\item[\text{$[n]$}] The inclusion map $(X_{n+1},\tau_{n+1})\hookrightarrow(X_{n},\tau_{n})$ is continuous. $T_n$ is well-defined and a strongly continuous and exponentially equicontinuous semigroup. We have $D(A_{n})=X_{n+1}$. $A_n^{-1}\colon(X_n,\tau_n)\rightarrow(X_n,\tau_n)$ exists and is continuous. $p_{\gamma,n+1}(\cdot)=p_{\gamma,n}(A_n(\cdot))$ holds for all $\gamma\in\Gamma$.
\end{compactitem}\vspace{-3pt}
(1) Since $A_{n+1}^{-1}\colon(X_{n+1},\tau_{n+1})\rightarrow(X_{n+1},\tau_{n+1})$ is continuous, for $\gamma\in\Gamma$ there exist $\delta\in\Gamma$ and $C\geqslant0$ such that $\check{p}_{\gamma,n}(x)=p_{\gamma,n+1}(A_{n+1}^{-1}x)\leqslant Cp_{\delta,n+1}(x)$ holds for any $x\in X_{n+1}$, i.e.~the identity $(X_{n+1},\tau_{n+1})\rightarrow(X_{n+1},\check{\tau}_n)$ is continuous and consequently the inclusion $(X_{n+1},\tau_{n+1})\hookrightarrow(X_{n},\tau_n)$ is continuous.
\smallskip
\\(2) Since $T_{n+1}$ is exponentially equicontinuous on $(X_{n+1},\tau_{n+1})$ there exists $\omega\in\mathbb{R}$ such that for $\gamma\in\Gamma$ there exist $\delta\in\Gamma$ and $M\geqslant1$ such that
\begin{eqnarray*}
\check{p}_{\gamma,n}(T_{n+1}(t)x)\hspace{-6pt}&=&\hspace{-6pt}p_{\gamma,n+1}(A_{n+1}^{-1}T_{n+1}(t)x)=p_{\gamma,n+1}(T_{n+1}(t)A_{n+1}^{-1}(t)x) \\
                                  &\leqslant&\hspace{-6pt}M\exp(\omega t)\,p_{\delta,n+1}(A_{n+1}^{-1}x)=M\exp(\omega t)\,\check{p}_{\delta,n}(x)
\end{eqnarray*}
holds for any $t\geqslant0$ and $x\in X_{n+1}$. Thus, $T_{n+1}(t)\colon(X_{n+1},\check{\tau}_{n})\rightarrow(X_{n+1},\check{\tau}_{n})$ is continuous for every $t\geqslant0$, i.e.~it extends continuously on the completions. Thus, $T_{n}(t)\colon(X_{n},\tau_{n})\rightarrow(X_{n+1},\tau_{n})$ is well defined. In particular, $T_n$ is exponentially equicontinuous and it satisfies the semigroup property. Finally, $T_n$ is strongly continuous on the dense subspace $X_{n+1}$ w.r.t.~the finer topology $\tau_{n+1}$ and thus also w.r.t.~$\check{\tau}_{n}$ and therefore by the density and the exponential equicontinuity also on the whole space $(X_{n},\tau_n)$.
\smallskip
\\(3) By definition $T_n(t)|_{X_{n+1}}=T_{n+1}(t)$ holds for any $t\geqslant0$. For $x\in D(A_{n+1})$ we have
$$
A_{n+1}x={\textstyle\lim_{t\searrow0}}{\textstyle\frac{1}{t}}(T_{n+1}(t)x-x)={\textstyle\lim_{t\searrow0}}{\textstyle\frac{1}{t}}(T_{n}(t)x-x)=A_nx.
$$
Thus, $D(A_{n+1})\subseteq D(A_n)$ and $A_{n}|_{D(A_{n+1})}=A_{n+1}$. Let $x\in X_{n+1}$. Then there exists a net $(x_i)_{i\in I}\subseteq D(A_{n+1})$ with $x_i\rightarrow x$ in $(X_{n+1},\tau_{n+1})$ and thus also in $(X_{n+1},\check{\tau}_{n})$. For $\gamma\in\Gamma$ and $i$, $j\in I$ we compute
$$
\check{p}_{\gamma,n}(A_{n}x_i-A_nx_j)=p_{\gamma,n+1}((A_{n+1}^{-1}A_{n+1})(x_i-x_j))=p_{\gamma,n+1}(x_i-x_j)
$$
which shows that $(A_nx_i)_{i\in I}$ is a Cauchy net in $(X_{n+1},\check{\tau}_n)$, since $(x_i)_{i\in I}$ is Cauchy in $(X_{n+1},\tau_{n+1})$. Therefore there exists $y\in X_{n}$ with $A_nx_i\rightarrow y$ in $(X_{n},\tau_n)$. Since $A_n$ is closed, we get $x\in D(A_n)$ which proves $X_{n+1}\subseteq D(A_n)$. By Lemma \ref{l1}.(ii), $X_{n+1}$ is dense in $(D(A_n),\tau_{A_{n}})$. Let us show that $\tau_{n+1}=\tau_{A_{n}}|_{X_{n+1}}$ holds. From this it follows that $(X_{n+1},\tau_{A_{n}}|_{X_{n+1}})$ is complete since this is true for $(X_{n+1},\tau_{n+1})$. Consequently, $X_{n+1}$ is at the same time dense and closed in $(D(A_{n}),\tau_{A_{n}})$ which implies $X_{n+1}=D(A_{n})$. Let $\alpha\in\Gamma$. There exist $\beta$ resp.~$\gamma\in\Gamma$ and $C\geqslant0$ such that 
\begin{eqnarray*}
p_{\alpha,A_{n}}(x)\hspace{-6pt}&=&\hspace{-6pt}\check{p}_{\alpha,n}(x) + \check{p}_{\alpha,n}(A_{n}x) =p_{\alpha,n+1}(A_{n+1}^{-1}x) + p_{\alpha,n+1}(A_{n+1}^{-1}A_{n+1}x) \\
                                &\leqslant&\hspace{-6pt}C p_{\beta,n+1}(x) + p_{\alpha,n+1}(x) \leqslant  C' p_{\gamma,n+1}(x)
\end{eqnarray*}
holds for $x\in D(A_{n+1})$. On the other hand for $\alpha\in\Gamma$ we have
\begin{eqnarray*}
p_{\alpha,n+1}(x)\hspace{-6pt}&\leqslant&\hspace{-6pt}p_{\alpha,n+1}(A_{n+1}^{-1}x) + p_{\alpha,n+1}(A_{n+1}^{-1}A_{n+1}x)\\
&=&\hspace{-6pt}\check{p}_{\alpha,n}(x) + \check{p}_{\alpha,n}(A_{n}x)=  p_{\alpha,A_n}(x)
\end{eqnarray*}
for $x\in D(A_{n+1})$. Thus, we have $\tau_{A_{n}}|_{D(A_{n+1})}=\tau_{n+1}|_{D(A_{n+1})}$. The first estimate above shows in addition that $\id_{X_{n+1}}|_{D(A_{n+1})}\colon(D(A_{n+1}),\tau_{n+1}|_{D(A_{n+1})})\rightarrow(X_{n+1},\tau_{A_n}|_{X_{n+1}})$ is continuous. Since $D(A_{n+1})\subseteq (X_{n+1},\tau_{n+1})$ is dense we get that  $\id_{X_{n+1}}\colon (X_{n+1},\tau_{n+1})\rightarrow(X_{n+1},\tau_{A_n}|_{X_{n+1}})$ is continuous. Thus, in the chain 
$$
D(A_{n+1})\subseteq (X_{n+1},\tau_{n+1})\subseteq (D(A_n),\tau_{A_n})
$$
we have that $D(A_{n+1})\subseteq (X_{n+1},\tau_{n+1})$ is dense, $\tau_{n+1}|_{D(A_{n+1})}=\tau_{A_n}|_{D(A_{n+1})}$ holds and $\tau_{n+1}$ is finer than $\tau_{A_n}|_{X_{n+1}}$. Thus, Bierstedt, Meise, Summers \cite[Lemma 1.2]{BMS1982} provides $\tau_{n+1}=\tau_{A_{n}}|_{X_{n+1}}$ and we obtain $D(A_n)=X_{n+1}$.
\smallskip
\\(4) By (3) we know that $A_n\colon(X_{n+1},\tau_{n+1})\rightarrow(X_n,\tau_n)$ is well-defined and continuous as the latter map is continuous for the graph topology. We know that $D(A_{n+1})=X_{n+2}$ holds; for $n=-1$ this is just the definition, for $n\leqslant-2$ this follows from our induction hypothesis. Therefore, $A_n|_{X_{n+2}}=A_{n}|_{D(A_{n+1})}=A_{n+1}$ and $X_{n+2}\subseteq X_{n+1}$ is dense, whence $A_n$ is the unique continuous extension of $A_{n+1}\colon(X_{n+2},\tau_{n+1}|_{X_{n+2}})\rightarrow(X_{n+1},\tau_{n}|_{X_{n+1}})$. The latter map is bijective; for $n=-1$ this follows from our general assumptions at the beginning of this section, for $n\leqslant-2$ it follows from the induction hypothesis. From the definition of $\check{\tau}_n$ it follows that $A_{n+1}\colon(X_{n+2},\tau_{n+1}|_{X_{n+2}})\rightarrow(X_{n+1},\tau_{n}|_{X_{n+1}})$ is also open, hence an isomorphism. Therefore, the continuous extension $A_n$ to the completions is also an isomorphism with inverse $A_n^{-1}\colon(X_{n},\tau_n)\rightarrow(X_{n+1},\tau_{n+1})$. In particular, (1) implies that $A_{n}^{-1}$ is continuous as a map from $(X_n,\tau_n)$ into itself.
\smallskip
\\(5) For $\gamma\in\Gamma$ and $x\in X_{n+2}$ we have $p_{\gamma,n+1}(x)=p_{\gamma,n+1}(A_{n+1}^{-1}A_{n+1}x)=\check{p}_{\gamma,n}(A_{n+1}x)=p_{\gamma,n}(A_{n+1}x)=p_{\gamma,n}(A_{n}x)$ since $A_{n+1}x\in X_{n+1}$ and $A_n|_{X_{n+2}}=A_{n+1}$ holds. As $X_{n+2}\subseteq (X_{n+1},\tau_{n+1})$ is dense and $p_{\gamma,n+1}$, $p_{\gamma,n}(A_n(\cdot))\colon (X_{n+1},\tau_{n+1})\rightarrow\mathbb{R}$ are continuous, the equality above also holds for all $x\in X_{n+1}$.
\end{proof}

\begin{cor}\label{c1}\begin{compactitem} \item[(i)] For every $n\in\mathbb{Z}$ the map $A_{n}^{-1}\colon(X_n,\tau_n)\rightarrow(X_{n+1},\tau_{n+1})$ is the restriction of $A_{n-1}^{-1}\colon(X_{n-1},\tau_{n-1})\rightarrow(X_{n},\tau_{n})$ and the unique continuous extension of $A_{n+1}^{-1}\colon(X_{n+1},\tau_{n+1})\rightarrow(X_{n+2},\tau_{n+2})$.
\vspace{3pt}

\item[(ii)] For every $n\in\mathbb{Z}$ and every $t\geqslant0$ we have $T_{n+1}(t)=A_n^{-1}T_n(t)A_n$, i.e.~$T_n$ and $T_{n+1}$ are \textit{similar}, cf.~\cite[II.5.3]{EngelNagelOne}. As in the classical setting we may visualize the \textit{Sobolev tower} with the following commutative diagram.
\begin{diagram}[height=2.0em,width=3.5em]
                                     & \vdots                     &                                 & \vdots                          & \\
                                     & \uTo                       &                                 & \dTo                            & \\
(X_n,\check{\tau}_{n-1})^{\hatt} =\;    & X_{n-1}                    & \rTo^{\scriptstyle T_{n-1}(t)}  & X_{n-1}                         & \\
                                     & \uTo^{\scriptstyle A_{n-1}}&                                 & \dTo_{\scriptstyle A_{n-1}^{-1}}& \\
(X_{n+1},\check{\tau}_{n})^{\hatt}=\;   & X_n                        & \rTo^{\scriptstyle T_n(t)}      & X_n                             & \; = \; (D(A_{n-1}),\tau_{n}) \\
                                     & \uTo^{\scriptstyle A_{n}}  &                                 & \dTo_{\scriptstyle A_{n}^{-1}}  & \\
(X_{n+2},\check{\tau}_{n+1})^{\hatt}=\;\;\;\;\, & X_{n+1}                    & \rTo^{\scriptstyle T_{n+1}(t)}  & X_{n+1}                         & \; = \; (D(A_n),\tau_{n+1}) \\
                                     & \uTo^{\scriptstyle A_{n+1}}&                                 & \dTo_{\scriptstyle A_{n+1}^{-1}}& \\
                                     & X_{n+2}                    & \rTo^{\scriptstyle T_{n+2}(t)}  & X_{n+2}                         & \;\;\;\; = \; (D(A_{n+1}),\tau_{n+2}) \\
                                     & \uTo                       &                                 & \dTo                            & \\
                                     & \rotatebox{90}{\dots}      &                                 & \rotatebox{90}{\dots}           & \\
                                     &                            &                                 &                                 & \\
\end{diagram}
\end{compactitem}\vspace{-30pt}
\end{cor}

Let us add some remarks on the construction.

\begin{rem}\label{r1}\begin{compactitem}\item[(i)] The proof of \cite[Proposition 1.3]{Komura} shows that the domain of the generator of a strongly continuous and exponentially equicontinuous semigroup is not only dense but sequentially dense in the underlying space. Therefore, we get that for every $n\in\mathbb{Z}$ the space $X_n$ is in fact sequentially dense in $(X_{n-1},\tau_{n-1})$, cf.~Theorem \ref{t1}.(i). In particular, $X_{n}$ can be thought of as the set of equivalence classes of (possibly non-convergent) $(X_n,\check{\tau}_{n+1})$-Cauchy sequences. Note, that in general this is not even true for the sequential completion, cf.~\cite[p.~146]{BPC}.
\vspace{3pt}

\item[(ii)] An inspection of the first part of the proof of Theorem \ref{t1} shows that the construction of the lower part of the Sobolev tower can also be performed when the starting space $(X_0,\tau_0)$ is only sequentially complete. The spaces $(X_n,\tau_n)$ for $n\geqslant1$ then will also be only sequentially complete, cf.~Lemma \ref{l1}.(i). For the upper tower this seems not to be true: In order to get $X_{n+1}=D(A_n)$ for $n\leqslant-1$ we showed that $X_{n+1}$ is dense and closed in $D(A_n)$ w.r.t.~the graph topology. If $(X_{n+1},\tau_{n+1})$ is only sequentially complete, our arguments show that $X_{n+1}$ is sequentially closed in $(D(A_n),\tau_{A_n})$. To show that the latter is dense we employed Lemma \ref{l1}.(ii), which is also valid if the underlying space is only sequentially complete. However, its proof (see \cite[Proposition 7]{AK}) does not provide that $D\subseteq(D(A),\tau_A)$ is sequentially dense since in general it might happen that $D\subset\seq(D)\subset\seq(\seq(D))$, where
$$
\seq(M)=\{x\in X\:;\:\exists\:(x_n)_{n\in\mathbb{N}}\subseteq M \colon x_n\rightarrow x \text{ for } n\rightarrow\infty\}
$$ 
for $M\subseteq (X,\tau)$. In this case $D$ then cannot be sequentially dense in $\seq(\seq(D))$.
\vspace{3pt}

\item[(iii)] The definition of the $X_n$ given in \cite[II.5]{EngelNagelOne} differs slightly from those given in \cite{Nagel}. In our approach above, the definition is completely symmetric in the sense that we only used $X_n$ to define $X_{n-1}$ for $n\geqslant0$ and we only used $X_{n+1}$ to define $X_{n}$ for $n\leqslant-1$. Therefore, the whole tower can be (uniquely) recovered from every single semigroup $T_n$ by construction. Given $T_m$, $X_m$ and $A_m$ we get by iteration that $X_n=D(A_m^{n-m})$ and $p_{\alpha,n}(\cdot)=p_{\alpha,m}(A_m^{n-m}(\cdot))$ holds for $n\geqslant m$. On the other hand for $n\leqslant m$ we get $(X_n,\tau_n)=(X_m,\check{\tau}_n)^{\hatt}$ where $\check{\tau}_n$ is given by the seminorms $\check{p}_{\alpha,n}(\cdot)=p_{\alpha,m}(A_m^{n-m}(\cdot))$. Note that above we use the conventions $A_m^0=\id_{X_m}$ and $D(\id_{X_m})=X_m$. The special case $m=0$ recovers the definitions in \cite[5.1]{EngelNagelOne} resp.~\cite[1.1 and 1.4]{Nagel}. Corresponding statements on $A_n$ and $T_
n$, e.g.~$T_n(t)=T_m(t)|_{X_n}$ for $n\geqslant m$ and $t\geqslant0$ etc., follow immediately.\vspace{3pt}

\item[(iv)] Let $\lambda\in\mathbb{C}$ be such that $A_0+\lambda\colon D(A_0)\rightarrow X_0$ is invertible and $(A_0+\lambda)^{-1}\in L(X_0)$ holds. Define $S_0=(S_0(t))_{t\geqslant0}$ via $S_0(t)=\exp(\lambda t)T_0(t)$ for $t\geqslant0$. By Lemma \ref{rescaling}, $(B_0,D(B_0))=(A_0+\lambda,D(A_0))$ is the generator of $S_0$. Let us for the Sobolev tower corresponding to $T_0$ keep the notation from above and denote the spaces in the tower corresponding to $S_0$ with $(Y_n,\sigma_n)$. With the same techniques as in the proof of Theorem \ref{t1} it can be shown that there are isomorphisms $\varphi_n\colon (X_n,\tau_n)\rightarrow(Y_n,\sigma_n)$ such that $\varphi_{n+1}=\varphi_{n}|_{X_{n+1}}$ holds for all $n\in\mathbb{Z}$, cf.~\cite[Exercise II.5.9]{EngelNagelOne}.
\end{compactitem}
\end{rem}

\vspace{-10pt}
Let us conclude this section with a concrete example for a Sobolev tower of Fr\'{e}chet spaces.

\begin{ex} Let $B=(b_{j,k})_{j,k\in\mathbb{N}}\subseteq[0,\infty[$ be a K\"othe matrix, see \cite[Section 27]{MeiseVogtEnglisch} for the precise definition, and let
$$
c_0(B)=\big\{x\in\mathbb{C}^{\mathbb{N}}\:;\:\forall\:k\in\mathbb{N}\colon\:{\textstyle\lim_{j\rightarrow\infty}}b_{j,k}|x_j|=0\big\}
$$
denote the K\"othe echelon space of order zero, which is a Fr\'{e}chet spaces for the seminorms $(p_k)_{k\in\mathbb{N}}$, $p_k(x)=\sup_{j\in\mathbb{N}}b_{j,k}|x_j|$ for $x\in c_0(B)$. Put $X_0=c_0(B)$. Let $(q_j)_{j\in\mathbb{N}}\subseteq \mathbb{C}$ with $\sup_{j\in\mathbb{N}}\Re q_j<0$. For $t\geqslant0$ put $T_0(t)x=(e^{tq_j}x_j)_{j\in\mathbb{N}}$, which defines a strongly continuous and equicontinuous semigroup $T_0=(T_0(t))_{t\geqslant0}$ on $X_0$ with generator $A_0x=(q_jx_j)_{j\in\mathbb{N}}$ for $x\in D(A_0)=\{x\in X\:;\:A_0x\in X_0\}$.
\smallskip
\\Similar to the Banach space case (cf.~\cite[II.5.7]{EngelNagelOne}) for every $n\in\mathbb{Z}$ we get
$$
X_n=c_0(B_n) \text{ with } B_n=(b_{j,k}|q_j^n|)_{j,k\in\mathbb{N}},\;\;T_n(t)x=(e^{tq_j}x_j)_{j\in\mathbb{N}} \text{ and } A_nx=(q_jx_j)_{j\in\mathbb{N}}
$$
for $x\in X_n$ and $t\geqslant0$ resp.~$x\in X_{n+1}$.
\end{ex}

\vspace{-5pt}
\section{Universal Extrapolation Spaces}\label{Extra}
\vspace{-10pt}
In the sequel we consider a Sobolev tower as constructed in Section \ref{S1}. By Theorem \ref{t1}, the spaces $(X_n,\tau_n)$ together with the inclusion maps $(X_n,\tau_n)\hookrightarrow(X_m,\tau_m)$ for $m\leqslant n$ form an inductive spectrum in the sense of Bierstedt \cite[p.~41]{Bierstedt1988}, where the ordering on the index set is the natural ordering but reversed. By Remark \ref{r1}.(iv) any rescaling that keeps the generator invertible yields an equivalent spectrum (cf.~Wengenroth \cite[dual to Definition 3.1.6]{Wengenroth}); the next definition is thus invariant up to an isomorphism under such rescalings. We denote by $(X,\tau)=\ind{n\in\mathbb{Z}}(X_n,\tau_n)$ the inductive limit and assume that it exists, cf.~the remarks in \cite[p.~42]{Bierstedt1988} and at the end of Section \ref{Intro}. Algebraically we have $X=\Bigcup{n\in\mathbb{Z}}X_n$ and the topology $\tau$ is by definition the finest locally convex topology on $X$ which makes all inclusions $(X_n,\tau_n)\hookrightarrow (X,\tau)$ 
continuous. It is easy to check that $(X,\tau)=\ind{n\in\mathbb{Z}_{\leqslant0}}(X_{n},\tau_{n})$ holds. If there is no risk of confusion, we will drop the topologies from notation and write $X=\ind{n}X_n$ for both representations given above. In the sequel we call the completion $\hat{X}$ of the inductive limit $X=\ind{n}X_n$ the \textit{universal extrapolation space} associated with the Sobolev tower. We will denote this space by $X_{-\infty}$ and its topology by $\tau_{-\infty}$.
\smallskip
\\For $t\geqslant0$, the maps $T_n(t)$ define an operator $T(t)\colon X\rightarrow X$. By the universal property of the inductive limit \cite[24.7]{MeiseVogtEnglisch}, the latter operators are continuous. Analogously, the maps $A_n\colon X_{n+1}\rightarrow X_n$ define a continuous operator $A\colon X\rightarrow X$. For $t\geqslant0$ let $T_{-\infty}(t)$ and $A_{-\infty}\colon X_{-\infty}\rightarrow X_{-\infty}$ denote the unique continuous extensions of $T(t)$ and $A$. We put $T_{-\infty}=(T_{-\infty}(t))_{t\geqslant0}$. Our aim is now to show that $T_{-\infty}$ is a strongly continuous and exponentially equicontinuous semigroup on $X_{-\infty}$ whose (continuous) generator is a restriction of $A_{-\infty}$. Under the assumption that $X_{-\infty}$ is barrelled, it will turn out that the generator is defined on the whole space $X_{-\infty}$, whence equals $A_{-\infty}$. For the proof we need some preparation; first we establish a locally convex version of the well-known product rule.

\begin{lem}\label{l2} Let $E$ and $F$ be complete locally convex spaces. Assume that $E$ is barrelled. For $a<b$ let $u\colon [a,b]\rightarrow E$ be differentiable and $G\colon [a,b]\rightarrow L(E,F)$ be strongly differentiable, i.e.~$G(\cdot)x\colon[a,b]\rightarrow F$, $s\mapsto G(s)x$ is differentiable for all $x\in E$. Then the map $g\colon [a,b]\rightarrow F$, $g(s)=G(s)u(s)$ is differentiable and we have
$$\textstyle
\frac{d}{ds}g(s)=\frac{d}{dr}[G(r)(u(s))]\big|_{r=s}+G(s)(\frac{d}{ds}u(s)).
$$
\end{lem}
\begin{proof} Let us fix $s\in[a,b]$ and $(h_n)_{n\in\mathbb{N}}\subseteq\mathbb{R}$ with $h_n\rightarrow0$ and $s+h_n\in [a,b]$ for all $n\in\mathbb{N}$. As in the proof of the classical product rule we compute
$$\textstyle
\frac{1}{h_n}(g(s+h_n)-g(s))=\big(\frac{G(s+h_n)-G(s)}{h_n}\big)u(s+h_n) +G(s)\big(\frac{u(s+h_n)-u(s)}{h_n}\big).
$$
For $n\rightarrow\infty$ the second summand converges to $G(s)(\frac{d}{ds}u(s))$, since $u$ is differentiable and $G$ is continuous, i.e.~$G(s)\in L(E,F)$ holds. We define $(H_n)_{n\in\mathbb{N}}\subseteq L(E,F)$ via $H_n(x)=\frac{G(s+h_n)-G(s)}{h_n}x$ for $x\in E$. By our assumptions $H_n\rightarrow H$ holds pointwise for $n\rightarrow\infty$ with $H\colon E\rightarrow F$, $Hx=\frac{d}{dr}G(r)x|_{r=s}$ for $x\in E$. Moreover, $(H_n)_{n\in\mathbb{N}}$ is pointwise bounded. Thus the Banach-Steinhaus theorem \cite[11.1.3]{Jarchow} implies $H\in L(E,F)$ and $H_n\rightarrow H$ uniformly on precompact subsets of $E$. As $u([a,b])\subseteq E$ is compact, for every continuous seminorm $p$ on $F$ and every $\epsilon>0$ there exists $N_1$ such that for all $n\geqslant N_1$ the estimate $p((H_n-H)(u(s+h_n)))<\epsilon/2$ is valid. For a given continuous seminorm $p$ on $F$ and $\epsilon>0$ we select $N_1$ as above. As $H$ is continuous, there exists a continuous seminorm $q$ on $E$ with $p(H(u(s+h_n)-u(s)))\leqslant q(
u(s+h_n)-u(s))$ for all $n$. Since $u$ is continuous, there exists $N_2\in\mathbb{N}$ such that $q(u(s+h_n)-u(s))<\epsilon/2$ holds for all $n\geqslant N_2$. With $N=\max(N_1,N_2)$ we get
\begin{eqnarray*}
p(H_n(u(s+h_n))-H(u(s)))&\leqslant&p((H_n-H)(u(s+h_n)))+p(H(u(s+h_n)-u(s)))\\
& \leqslant& p((H_n-H)(u(s+h_n)))+q(u(s+h_n)-u(s))\\
&\leqslant&\epsilon/2+\epsilon/2=\epsilon
\end{eqnarray*}
for $n\geqslant N$. This shows that the first summand in the first equation of this proof converges to $H(u(s))=\frac{d}{dr}[G(r)(u(s))]\big|_{r=s}$. Since $(h_n)_{n\in\mathbb{N}}$ and $s$ were arbitrary this finishes the proof.
\end{proof}

\vspace{-10pt}
Next we use Lemma \ref{l2} to show that a generator which can be extended continuously to the whole space has already the whole space as domain. For Banach spaces this is well-known \cite[II.1.5]{EngelNagelOne}; in the case of a locally convex space we however have to assume that the underlying space is barrelled.

\begin{lem}\label{l3} Let $E$ be a complete locally convex space and $S=(S(t))_{t\geqslant0}$ be a strongly continuous and exponentially equicontinuous semigroup with generator $(B,D(B))$. Let $C\in L(E)$ be such that $C|_{D(B)}=B$. Then $(C,E)$ is the generator of a strongly continuous and exponentially equicontinuous semigroup $R=(R(t))_{t\geqslant0}$. If $E$ is barrelled, then $R=S$ and thus $(B,D(B))=(C,E)$ holds.
\end{lem}
\vspace{-20pt}
\begin{proof} From Choe \cite[Corollary 4.5 and condition (G) on p.~295]{Choe} it follows that there exists $\omega\geqslant0$ such that $\{(\lambda-\omega)^k(\lambda-B)^{-k}\:;\:\lambda>\omega,k=0,1,2,\dots\}\subseteq L(E)$ is equicontinuous. In particular, $\lambda-B\colon D(B)\rightarrow E$ is invertible with $(\lambda-B)^{-1}\in L(E)$. By our assumptions, we know that $\lambda-C$ belongs to $L(E)$ and that $(\lambda-C)|_{D(B)}=\lambda-B$ holds. It follows $(\lambda-C)(\lambda-B)^{-1}=\id_E$ and $(\lambda-B)^{-1}(\lambda-C)=\id_{D(B)}$. Since $(\lambda-B)^{-1}$ and $\lambda-C$ are in $L(E)$ and $D(B)$ is dense in $E$, we get that $\lambda-C\colon E\rightarrow E$ is invertible with the continuous inverse $(\lambda-C)^{-1}=(\lambda-B)^{-1}$. Thus also $(\lambda-C)^{-k}=(\lambda-B)^{-k}$ for $k=0,1,2,\dots$ holds. By \cite[Corollary 4.5 and condition (G) on p.~295]{Choe} it follows that $(C,E)$ is the generator of a strongly continuous and exponentially equicontinuous semigroup $R=(R(t))_{t\geqslant0}$.
\smallskip
\\In order to show $R=S$ we fix $t>0$ and $x\in D(B)$. We define the map $g\colon[0,t]\rightarrow E$, $s\mapsto R(t-s)S(s)x$. By Lemma \ref{l2} it follows (with $u(s)=S(s)x$ and $G=R(t-\cdot)$) that $g$ is differentiable. By \cite[Proposition 1.2]{Komura} and since $B$ generates $S$ and $C$ generates $R$ we may compute
\begin{eqnarray*}
\textstyle\frac{d}{ds}g(s)&=&\textstyle\frac{d}{dr}[R(t-r)(S(s)x)]\big|_{r=s}+R(t-s)(\frac{d}{ds}S(s)x)\\
&=&-R(t-r)CS(s)x\big|_{r=s}+R(t-s)S(s)Bx\\
&=&-R(t-s)S(s)Bx+R(t-s)S(s)Bx
\end{eqnarray*}
where we used $CS(s)x=BS(s)x=S(s)x$ since $S(s)x$ and $x$ belong to $D(B)$. Now we compute
$$\textstyle
S(t)x-R(t)x=\int_0^t\frac{d}{ds}[R(t-s)S(s)x]ds=\int_0^tg(s)ds=0.
$$
Since $S(t)-R(t)\in L(E)$ and $D(B)\subseteq E$ is dense we get $S(t)=R(t)$ on $E$. As $t>0$ was arbitrary, $S=R$ follows. Consequently, $(B,D(B))=(C,E)$.
\end{proof}

\vspace{-10pt}
Let us remark that the two lemmas above are also valid, if we only assume that the underlying spaces are sequentially complete. Thus, the statement of this section's main theorem below remains true if we define $X_{-\infty}$ to be the sequential completion (instead of the completion) of the inductive limit $X$.

\begin{thm}\label{t2} Consider a Sobolev tower as in Section \ref{S1} and the associated universal extrapolation space as defined above.\vspace{-5pt}
\begin{compactitem}
\item[(i)] $T_{-\infty}$ is a strongly continuous and exponentially equicontinuous semigroup on $X_{-\infty}$. Its generator $(G,D(G))$ satisfies $X\subseteq D(G)\subseteq X_{-\infty}$ and $A_{-\infty}|_{D(G)}=G$. In particular, $G\colon (D(G),\tau_{-\infty})\rightarrow X_{-\infty}$ is continuous.\vspace{3pt}
\item[(ii)] If $X_{-\infty}$ is barrelled, then $D(G)=X_{-\infty}$ and consequently $G=A_{-\infty}$ holds.\vspace{3pt}
\item[(iii)] The operator $A_{-\infty}\colon X_{-\infty}\rightarrow X_{-\infty}$ is an isomorphism. Its inverse $A_{-\infty}^{-1}$ is the unique continuous extension of the map $A^{-1}\colon X\rightarrow X$ induced by the inverses $A_{n}^{-1}\colon X_n\rightarrow X_{n+1}$ of the operators $A_n$.\vspace{3pt}
\end{compactitem}
\end{thm}
\vspace{-20pt}
\begin{proof} (i) $T=(T(t))_{t\geqslant0}$ and thus also $T_{-\infty}$ satisfies the semigroup property. For $x\in X$ we select $n$ such that $x\in X_n$. Then $T(\cdot)x\colon[0,\infty[\:\rightarrow X_n$ is continuous and since the inclusion $X_n\hookrightarrow X$ is continuous it follows that $T$ is strongly continuous. We show that $T$ is exponentially equicontinuous. The proof of Theorem \ref{t1} shows that there is $\omega\in\mathbb{R}$ such that $\{\exp(\omega t)T_n(t)\:;\:t\geqslant0\}$ is equicontinuous for any $n\in\mathbb{Z}$. Denote by $S=\exp(\omega\,\cdot\,)T$ resp.~$S_n=\exp(\omega\,\cdot\,)T_n$ the rescaled semigroups. Let $U$ be a 0-neighborhood in $X$. W.l.o.g.~we may assume that $U$ is absolutely convex. By \cite[p.~43]{Bierstedt1988} this implies that $U\cap X_n$ is a 0-neighborhood in $X_n$ for all $n$. Since every $S_n$ is equicontinuous there exist absolutely convex 0-neighborhoods $V_n$ in $X_n$ such that $[S_n(t)](V_n)\subseteq U\cap X_n$ holds for every $n$ and every $t\geqslant0$. We 
put $V=\abs(\,\Bigcup{n\in\mathbb{N}}V_n)\subseteq X$ which is a 0-neighborhood in $X$, see \cite[p.~43]{Bierstedt1988}. But then we have
$$
[S(t)](V)=[S(t)]\big(\abs(\,\Bigcup{n\in\mathbb{N}}V_n)\big)=\abs\big(\Bigcup{n\in\mathbb{N}}\,[S_n(t)](V_n)\big)\subseteq \abs\big(\Bigcup{n\in\mathbb{N}}\,(U\cap X_n)\big)=U
$$
for all $t\geqslant 0$, i.e.~$S$ is equicontinuous by \cite[p.~156]{Jarchow} and whence $T$ is exponentially equicontinuous. As $X\subseteq X_{-\infty}$ is dense, it firstly follows that also $T_{-\infty}$ is exponentially equicontinuous. Secondly, the density, the exponential equicontinuity and the strong continuity of $T$ imply that $T_{-\infty}$ is strongly continuous. Denote by $(G,D(G))$ the generator of $T_{-\infty}$. For $x\in X$ we select $n$ such that $x\in X_n=D(A_{n-1})\subseteq X_{n-1}$ holds. Then
$$
Ax=A_{n-1}x=\lim_{t\searrow0}{\textstyle\frac{T_n(t)x-x}{t}}=\lim_{t\searrow0}{\textstyle\frac{T(t)x-x}{t}}
$$
follows. In particular the limits exist in $X$ and thus $x\in D(G)$ with $Ax=Gx$ holds. Thus, $X\subseteq D(G)$ and $A=G|_{X}$ are true. As $X$ is $T_{-\infty}$-invariant and dense in $(D(G),\tau_{-\infty}|_{D(G)})$, Lemma \ref{l1} implies that $X\subseteq (D(G),\tau_G)$ is dense. Since the inclusion $(D(G),\tau_G)\rightarrow X_{-\infty}$ is continuous, $A_{-\infty}|_{D(G)}\colon (D(G),\tau_G)\rightarrow X_{-\infty}$ is continuous. $A_{-\infty}|_X=A=G|_{X}$ thus implies by the density of $X$ that $A_{-\infty}|_{D(G)}=G$. It follows that $G=A_{-\infty}|_{D(G)}\colon (D(G),\tau_{-\infty})\rightarrow X_{-\infty}$ is continuous.
\medskip
\\(ii) The conclusion follows immediately from (i) and Lemma \ref{l3}.
\medskip
\\(iii) The maps $A^{-1}_n\colon X_n\rightarrow X_{n+1}$ induce a map $A^{-1}\colon X\rightarrow X$ which is continuous by the universal property of the inductive limit. Given $x\in X$ select $n$ such that $x\in X_{n}$. Then $A^{-1}Ax=A^{-1}A_{n-1}x=A^{-1}_{n-1}A_{n-1}x=x$. Conversely, $AA^{-1}x=AA^{-1}_nx=A_nA^{-1}_nx=x$. Thus, $A\colon X\rightarrow X$ is an isomorphism with inverse $A^{-1}$. It follows that $A_{-\infty}\colon X_{-\infty}\rightarrow X_{-\infty}$ is also an isomorphism and that its inverse $A_{-\infty}^{-1}$ is the unique extension of $A^{-1}$ to $X_{-\infty}$.
\end{proof}

\vspace{-10pt}
On a first glance, the assumption in Theorem \ref{t2}.(ii) seems to make the result rather unhandy as one has first to compute $X_{-\infty}$ and then to check if it is barrelled. Let us thus mention that the latter holds at least if $X=\ind{n}X_n$ is barrelled, see \cite[4.2.1]{BPC}. This in turn is valid if all steps $X_n$ are barrelled, see \cite[4.2.6]{BPC}. In particular, in the setting of classical Sobolev towers the extra assumption of Theorem \ref{t2}.(ii) can be omitted, see Theorem \ref{thmA}.A.
\smallskip
\\Let us add that even in the rather pleasant situation of Theorem \ref{thmA}.A it is not clear if the inductive limit $X$ is complete, cf.~\cite{Bierstedt1988} and its list of references for more information on the completeness of LB-spaces. We now give an example for the situation of Theorem \ref{thmA}.A.

\begin{ex}\label{ex2} Let $X_0=\ell^2$. For $t\geqslant0$ and $x\in\ell^2$ we put $T_0(t)x=(e^{tq_j}x_j)_{j\in\mathbb{N}}$ where $(q_j)_{j\in\mathbb{N}}\subseteq\mathbb{C}$ is a sequence with $\sup_{j\geqslant0}\Re(q_j)<0$ and $|q_j|\geqslant1$ for all $j\in\mathbb{N}$. Then, $T_0=(T_0(t))_{t\geqslant0}$ defines a strongly continuous semigroup on $\ell^2$ whose generator $A_0$ is given by $A_0x=(q_jx_j)_{j\in\mathbb{N}}$ defined on its maximal domain. The Sobolev tower is given by the following data. We have
$$
X_n=\big\{x\in\mathbb{C}^{\mathbb{N}}\:;\:\|x\|_n=\big(\Bigsum{j=0}{\infty} (|q_j|^n|x_j|)^2\big)^{1/2}<\infty\big\},
$$
$T_n$ and $A_n$ are given by the formulas for $T$ above with the appropriate domains. We put $v_{n,j}=|q_j|^n$ and $v_n=(v_{n,j})_{j\in\mathbb{N}}$. Then we have $X_n=\ell^2(v_n)$. Since $v_{n,j}=|q_j|^{n}\leqslant |q_j|^{n+1}=v_{n+1,j}$ for any $j\in\mathbb{N}$ and $n\leqslant-1$ and taking into account the inverse ordering of our index set we obtain that $\ind{n\leqslant0}X_n=\ind{n\leqslant0}\ell^2(v_n)=k^2(V)$ is the K\"othe coechelon space of order two in the notation of \cite[Section 1]{BMS1982a} for $V=(v_n)_{n\in\mathbb{N}}$. Since the latter space is complete, see \cite[2.3]{BMS1982a}, it follows $X_{-\infty}=k^2(V)$. Let us add that in \cite[2.3]{BMS1982a} a fundamental system of seminorms for the topology of $k^2(V)$ is explicitly given.
\end{ex}

\vspace{-5pt}
\section{The dual construction: Universal Interpolation Spaces}\label{Inter}
\vspace{-10pt}
Dual to the construction of Section \ref{Extra} one can consider the projective limit of the spaces $X_n$ in a given Sobolev tower. Let us call $(X_{\infty},\tau_{\infty})=\proj{n}(X_n,\tau_n)$ the \textit{universal interpolation space} associated with a Sobolev tower as in Section \ref{S1}. This time the ordering of the index set is just the usual one of the integers. As in Section \ref{Extra}, we may restrict our attention to the spaces with non-negative indices and appropriate rescalings yield equivalent spectra and thus isomorphic projective limits. Algebraically, we have $X_{\infty}=\Bigcap{n\geqslant0}X_n=\Bigcap{n\geqslant0}D(A^n)$, cf.~Remark \ref{r1}.(iii). The topology $\tau_{\infty}$ of $X_{\infty}$ is the coarsest locally convex topology which makes all inclusions $X_{\infty}\rightarrow X_n$ continuous. In contrast to the inductive case this topology always exists and in addition it is complete as the $X_n$ are all complete. Copying the proof of \cite[II.1.8]{EngelNagelOne} verbatim, it follows 
that $X_{\infty}=\proj{n}X_n$ is a reduced projective limit in the sense of Floret, Wloka \cite[p.~143]{FloretWloka}, i.e.~$X_{\infty}\subseteq X_n$ is dense for any $n\geqslant0$. 
\smallskip
\\On $X_{\infty}$ we define $T_{\infty}=(T_{\infty}(t))_{t\geqslant0}$ via $T_{\infty}(t)x=T_0(t)x$ for $t\geqslant0$ and $A_{\infty}$ via $A_{\infty}x=A_0x$ for $x\in X_{\infty}$. Both maps are well-defined, $T_{\infty}$ satisfies the semigroup property and $A_{\infty}\colon X_{\infty}\rightarrow X_{\infty}$ is continuous by the universal property of the projective limit. We define $A_{\infty}^{-1}$ via $A_{\infty}^{-1}x=A_0^{-1}x$. Again by the universal property, $A_{\infty}^{-1}$ is continuous and it follows that $A_{\infty}$ is an isomorphism with inverse $A_{\infty}^{-1}$. We have the following result dual to Theorem \ref{t2}.

\begin{thm}\label{t3} Consider a Sobolev tower as in Section \ref{S1} and let $X_{\infty}$ be the associated universal interpolation space. $X_{\infty}=\proj{n}X_n$ is a reduced projective limit and $X_{\infty}$ is complete.  $T_{\infty}$ is a strongly continuous and exponentially equicontinuous semigroup on $X_{\infty}$ with generator $(A_{\infty},X_{\infty})$. In particular, $A_{\infty}\colon X_{\infty}\rightarrow X_{\infty}$ is an isomorphism; its inverse is given as the restriction of the inverse of $A_n$ on any step.
\end{thm}
\vspace{-20pt}
\begin{proof} $T_{\infty}$ is strongly continuous by the universal property of the projective limit. Let $U$ be a $0$-neighborhood in $X_{\infty}$. Then there exists $n\geqslant0$ and a $0$-neighborhood $U_n$ in $X_n$ such that $X_{\infty}\cap U_n\subseteq U$ holds, see \cite[2.6.1.(b)]{Jarchow}. We select $\omega\in\mathbb{R}$ such that $\{\exp(\omega t)T_n(t)\:;\:t\geqslant0\}$ is equicontinuous for any $n\in\mathbb{Z}$ and denote by $S_n=\exp(\omega\,\cdot\,)T_n$ for $n\in\mathbb{Z}\cup\{-\infty\}$ the rescaled semigroups. By the equicontinuity there exists a $0$-neighborhood $V_n$ in $X_n$ such that $S_n(t)(V_n)\subseteq U_n$ holds for all $t\geqslant0$. Put $V=X_{\infty}\cap V_n$ which is a $0$-neighborhood in $X_{\infty}$. Let $t\geqslant0$. On the one hand we have $S_{\infty}(t)(V)=S_n(t)(X_{\infty}\cap V_n)\subseteq S_n(t)(V_n)\subseteq U_n$. On the other hand, $S_{\infty}(t)(V)=S_{\infty}(t)(X_{\infty}\cap V_n)\subseteq S_{\infty}(t)(X_{\infty})\subseteq X_{\infty}$ holds and thus $S_{\infty}(t)(V) \subseteq X_{\infty}\cap U_n\subseteq U$ follows. As $t\geqslant0$ was arbitrary, $S_{\infty}$ is equicontinuous and consequently $T_{\infty}$ is exponentially equicontinuous.
\smallskip
\\Let $x\in X_{\infty}$ be arbitrary. Let $n\in\mathbb{N}$ be given. By definition, $x\in D(A_n)\subseteq X_n$ holds and we have
$$
\lim_{t\searrow0}{\textstyle\frac{T_{\infty}(t)x-x}{t}}=\lim_{t\searrow0}{\textstyle\frac{T_n(t)x-x}{t}}=A_nx=A_{\infty}x
$$ 
where the limit exists in the topology of $X_n$. Thus, $(A_{\infty},X_{\infty})$ is the generator of $T_{\infty}$.
\end{proof}

\vspace{-10pt}
Theorem \ref{t3} allows for the following variant: If the starting space $X_0$ is only sequentially complete, then by Remark \ref{r1}.(ii) the lower Sobolev tower can be constructed analogously to our approach in Section \ref{S1} and consists of sequentially complete spaces. We can thus define $X_{\infty}=\proj{n\geqslant0} X_n$ as the projective limit over the lower tower and obtain a sequentially complete locally convex space, cf.~\cite[3.3.7]{Jarchow}. $T_{\infty}$, $A_{\infty}$ and $A_{\infty}^{-1}$ can be defined then as in the setting above and the conclusions of Theorem \ref{t3} all remain valid.
\smallskip
\\As in the inductive case, we stated the classical situation of a strongly continuous semigroup acting on a Banach space already at the beginning, see Theorem \ref{thmB}.B.
\smallskip
\\Finally, we continue Example \ref{ex2} and compute the universal interpolation space. Note that the projective limit of the classical Sobolev spaces, $\mathcal{D}_{L_2}(\mathbb{R}^n)$, is investigated in \cite[\S{}14]{MeiseVogtEnglisch}.

\begin{ex}\label{ex3} In the notation of Example \ref{ex2} we define $B=(b_n)_{n\in\mathbb{N}}$ via $b_n=(b_{n,j})_{j\in\mathbb{N}}$, $b_{n,j}=|q_j|^n$. Note that $b_{n,j}=|q_j|^n\geqslant |q_j|^{n+1}=b_{n+1}$ and that we now use the usual ordering of $\mathbb{N}$ for our index set. Then, $X_{\infty}=\proj{n\geqslant0}X_n=\proj{n\geqslant0}\ell^2(b_n)=\lambda^2(B)$, i.e.~the universal interpolation space is the K\"othe echelon space of order two associated to the K\"othe matrix $B$; the topology of this Fr\'{e}chet space is given by the weighted seminorms $p_n(x)=[\Bigsum{j=1}{\infty}(b_{n,j}|x_j|)^2]^{1/2}$ for $x\in\lambda^2(B)$ and $n\geqslant0$.
\end{ex}

\begin{center}
\textbf{Acknowledgement}
\end{center}\vspace{-18pt}
The author likes to thank B.~Jacob for several fruitful discussions on the topic of this note. Moreover, he likes to thank B.~Farkas, who drew the author's attention to the study of inductive and projective limits of Sobolev towers. Finally, the author likes to thank the referee for his careful work and his valuable comments.

\setlength{\parskip}{0cm}

\end{document}